\documentclass[11pt,thmsa]{article}%
\usepackage{amsmath}
\usepackage{amsfonts}
\usepackage{amssymb}
\usepackage{graphicx}%
\newtheorem{theorem}{Theorem}[section]

\newtheorem{lemma}[theorem]{Lemma}

\newtheorem{proposition}[theorem]{Proposition}

\newenvironment{proof}[1][Proof]{\noindent\textbf{#1.} }{\ \rule{0.5em}{0.5em}}

\begin{document}
\title{Homogeneous $(\alpha,\beta)$-spaces with positive flag curvature and vanishing S-curvature \thanks{Project supported by NSFC (no. 11271216, 11271198, 11221091),State Scholarship Fund of CSC (no. 201408120020), Doctor fund of Tianjin Normal University (no. 52XB1305) and SRFDP of China}}
\author{Ming Xu$^1$ and Shaoqiang Deng$^2$ \thanks{Corresponding author. E-mail: dengsq@nankai.edu.cn}\\
\\
$^1$College of Mathematics\\
Tianjin Normal University\\
 Tianjin 300387, P.R. China\\
 \\
$^2$School of Mathematical Sciences and LPMC\\
Nankai University\\
Tianjin 300071, P.R. China}
\date{}
\maketitle
\begin{abstract}
We give a complete classification of homogeneous $(\alpha,\beta)$-metrics  with positive flag curvature and vanishing S-curvature.



\end{abstract}
\section{Introduction}
In this paper we give a complete classification of homogeneous $(\alpha,\beta)$-metrics with positive flag curvature and vanishing S-curvature. One of the central and most difficult problems
in Riemannian geometry is the classification of smooth manifolds which admit Riemannian metrics with positive sectional curvature. In the homogeneous case, Berger and Wilking \cite{Ber, WI} obtained a classification of
normal homogeneous Riemannian manifolds with positive curvature. Wallach \cite{Wa} gave a classification  of
even-dimensional homogeneous Riemannian manifolds with positive curvature. Then Aloff-Wallach \cite{AloffWallach1975} constructed an infinite  series of  $7$-dimensional homogeneous Riemannian manifolds with positive curvature which has distinct homotopy type. In 1976,  Brad Bergery \cite{BB} made an argument to show that the lists of Berger-Wilking, Wallach and Aloff-Wallach exhaust all the homogeneous  Riemannian manifolds with positive sectional curvature. This gives a complete classification of positively curved homogeneous Riemannian manifolds. Note that recently the first-named author and Wolf \cite{XW} found that there is a  substantial gap in Bergery's argument. This error has been corrected. Therefore, the classification of positively curved homogeneous Riemannian manifolds  is complete.

  In Finsler geometry,  the study of Finsler spaces with positive flag curvature possesses the same importance as in Riemannian geometry. Therefore, it is a key problem to classify homogeneous Finsler spaces with positive flag curvature. In some early considerations, the authors of this paper and collaborators have obtained a classification of even-dimensional homogeneous Finsler spaces, as well as that of homogeneous Randers spaces with
  positive flag curvature and vanishing S-curvature. However, as stated before, the  classification of the odd-dimensional case is much harder and up to now we have not found an efficient technique to solve this problem. The main result of this paper can be viewed as a step to the final settlement of this problem.

  The main theorem of this paper can be stated as the following:

\begin{theorem}\label{main1}
Let $G$ be a connected simply connected Lie group and $H$ be a closed subgroup of $G$. Suppose on the coset space $G/H$ there exists a $G$-invariant non-Riemannian $(\alpha,\beta)$-metric with positive flag curvature and vanishing S-curvature. Then the coset space $G/H$ must be equivalent to (in the sense that the corresponding Lie algebra pair $(\mathfrak{g},\mathfrak{h})$ is the same) is  one of the following:
\begin{description}
\item{\rm (1)} The homogeneous spheres $S^{2n+1}=\mathrm{SU}(n+1)/\mathrm{SU}(n)$, $S^{2n+1}
    =\mathrm{U}(n+1)/\mathrm{U}(n)$, $S^{4n+3}=\mathrm{Sp}(n+1)/\mathrm{Sp}(n)$, $S^{4n+3}=(\mathrm{Sp}(n+1)\mathrm{U}(1))/(\mathrm{Sp}(n)\times\mathrm{U}(1))$;
\item{\rm (2)} The Aloff-Wallach spaces $S_{k,l}$ with $kl(k+l)\neq 0$;
\item{\rm (3)} The coset space $\mathrm{U}(3)/{T^2}$, where $T^2$ is a two dimensional torus in the standard maximal torus of $\mathrm{U}(3)$ consisting of diagonal matrices, and it is not contained in the standard subgroup $\mathrm{SU}(3)$.
\end{description}
Moreover, on any of above coset spaces there exists an invariant non-Riemannian $(\alpha,\beta)$-metric with positive flag curvature and vanishing S-curvature.
\end{theorem}

We remark here that generally it is very hard to get a complete classification of all the invariant  $(\alpha,\beta)$-metrics with positive flag curvature and vanishing S-curvature  on the above coset spaces under isometries. In fact, this is even very difficult for Riemannian metrics; see \cite{VZ} for some information on spheres.

  In Section 2, we recall some preliminaries on Finsler spaces, $(\alpha,\beta)$-metrics, geodesic spray, Riemann curvature and S-curvature. In Section 3, we give a characterization of homogeneous $(\alpha,\beta)$-metrics with positive flag curvature and vanishing S-curvature.
  Section 4 is devoted to determining all the possible coset spaces admitting invariant $(\alpha,\beta)$-metrics with positive flag curvature and vanishing S-curvature.
  Finally, in Section 5, we complete the proof of the main theorem.

\section{Preliminaries}
In this section, we recall some fundamental notions and facts in Finsler geometry, especially those relevant to geodesic spray and curvatures in Finsler geometry. For details, the readers are referred to \cite{CS}.
\subsection{Minkowski norm and Finsler metric}
A {\it Minkowski norm} on a real vector space $\mathbf{V}$, $\dim\mathbf{V}=n$, is a
continuous function $F:\mathbf{V}\rightarrow[0,+\infty)$ satisfying the following
conditions:
\begin{description}
\item{\rm (1)} $F$ is positive and smooth on $\mathbf{V}\backslash\{0\}$;
\item{\rm (2)} $F(\lambda y)=\lambda F(y)$ for any $\lambda>0$;
\item{\rm (3)} With respect to any linear coordinates $y=y^i e_i$, the Hessian matrix
\begin{equation}
(g_{ij}(y))=\left(\frac{1}{2}[F^2]_{y^i y^j}(y)\right)
\end{equation}
is positive definite at any $y\ne 0$.
\end{description}
The Hessian matrix $(g_{ij}(y))$ and its inverse $(g^{ij}(y))$ of a Minkowski norm
can be used to raise up or lower down indices of tensors on the vector space.

Given any $y\neq 0$, the Hessian matrix $(g_{ij}(y))$ defines an inner product
$\langle \cdot,\cdot\rangle_y$ (or $\langle \cdot,\cdot\rangle_y^F$ when $F$ needs to be
specified) on $\mathbf{V}$ by
$\langle u,v \rangle_y=g_{ij}(y)u^i v^j$, where $u=u^i e_i$ and $v=v^i e_i$. This inner product can also be
written as
\begin{equation}
\langle u,v\rangle_{y}=\frac{1}{2}\frac{\partial^2}{\partial s
\partial t}[F^2(y+su+tv)]|_{s=t=0},
\end{equation}
which is independent of the choice of the linear coordinates.

A Finsler metric $F$  on a smooth manifold $M$, $\dim M=n$, is a continuous function
$F:TM\rightarrow [0,+\infty)$ such that it is positive and smooth on the slit
tangent bundle $TM\backslash 0$, and its restriction to each tangent space is
a Minkowski norm. We generally say that  $(M,F)$ is a {\it Finsler manifold} or
{\it Finsler space}.

Important examples of Finsler metrics include
Riemannian metrics, Randers metrics, $(\alpha,\beta)$-metrics, etc.
Riemannian metrics are a special class of Finsler metrics whose
 Hessian matrices at each point only depends only on $x\in M$.
A Riemannian  metric can also be defined as a  global smooth section
$g_{ij}dx^i dx^j$ of $\mathrm{Sym}^2(T^* M)$.
Randers metrics are the  simplest and important class of non-Riemannian metrics in Finsler geometry, which are Finsler metrics of the form $F=\alpha+\beta$, in which $\alpha$ is a Riemannian metric and $\beta$ is
a 1-form, where the norm of $\beta$ with respect to $\alpha$ is everywhere $<1$. Randers metrics can be naturally generalized to $(\alpha,\beta)$-metrics which have
the form $F=\alpha\phi(\beta/\alpha)$, where  $\phi$ is a smooth positive function (see \cite{CS}).
The condition for an $(\alpha,\beta)$-metric to be positive definite can be expressed as:
\begin{equation}\label{(2.3)}
\phi (s)-s\phi '(s)+(b_x^2-s^2)\phi''(s)>0,
\end{equation}
for any $x\in M$ and $||s||\leq b_x$, whenever $b_x=||\beta (x)||_{\alpha(x)}$ is positive (see \cite{CS}).
In recent years, much work has been done on Randers metrics and   $(\alpha,\beta)$-metrics.

\subsection{Geodesic spray, geodesics and S-curvature}

 To study the local geometric properties of a Finsler space $(M,F)$, we usually use the {\it standard} local coordinates $(x^i,y^j)$, where $x=(x^i)\in M$ is a local coordinate system on
 an open subset $U$ of $M$,
and $y=y^j\partial_{x^j}\in T_x M$. It is easily seen that $(x^i, y^i)$ defines a local coordinate system on $TU$.

The geodesic spray $G$ is a globally defined   smooth vector field   on $TM\backslash 0$. On a standard
local coordinate system, it can be given as
\begin{equation}
G=y^i\partial_{x^i}-2G^i\partial_{y^i},
\end{equation}
in which
\begin{equation}
G^i=\frac{1}{4} g^{il}([F^2]_{x^k y^l}y^k - [F^2]_{x^l}).
\end{equation}

A curve $c(t)$ on $M$ is called a geodesic if $(c(t),\dot{c}(t))$ is an integration curve
of $G$, in which the tangent field $\dot{c}(t)=\frac{d}{dt}c(t)$ along the curve gives
the speed. For standard local coordinates, a geodesic $c(t)=(c^i(t))$ satisfies the equations
\begin{equation}
\ddot{c}^i(t)+2 G^i(c(t),\dot{c}(t))=0.
\end{equation}

It is well known that the geodesic spray is tangent to the indicatrix bundle in $TM$. Thus
$F(\dot{c}(t))$ is a constant function when $c(t)$ is a geodesic, and
we only consider geodesics of nonzero constant speed.

Z. Shen defines the following important non-Riemannian curvature using the geodesic spray, which is now generally called  S-curvature in the literature.

Let $(M,F)$ be a Finsler space. Given  any local coordinate system, the Busemann-Hausdorff volume form can be
defined as $dV_{\mathrm{BH}}=\sigma(x)dx^1\cdots dx^n$, where
$$
\sigma(x)=\frac{\omega_n}{\mbox{Vol}\{(y^i)\in\mathbb{R}^n|F(x,y^i\partial_{x^i})<1\}},
$$
in which $\mbox{Vol}$ denotes the volume of a subset with respect to the standard Euclidian metric on
$\mathbb{R}^n$, and $\omega_n=\mbox{Vol}(B_n(1))$. It is easily seen that the Busemann-Hausdorff
form is globally defined and does not depend on the specific coordinate system. On the other hand,
although the coefficient function $\sigma(x)$ is only locally defined and
depends on the choice of local coordinates $x=(x^i)$, the distortion function
\begin{equation*}
\tau(x,y)=\ln\frac{\sqrt{\det(g_{ij}(x,y))}}{\sigma(x)}
\end{equation*}
on $TM\backslash 0$ is independent of the local coordinates and is globally defined.

The S-curvature $S(x,y)$ on $TM\backslash 0$ is defined as the derivative of
$\tau(x,y)$ in the direction of the geodesic spray $G(x,y)$.

\subsection{Riemannian curvature and flag curvature}
The Riemannian curvature of  a Finsler manifold is the natural generalization of the relevant  curvature  in the Riemannian case.  It can be defined either by the Jacobi field
or the structure equation for the curvature of the Chern connection. Now let us recall the definition of this important quantity.

Given  a standard local coordinate system and a non-zero vector $y\in T_xM$,
the Riemannian curvature at $y$ is defined  as a linear endomorphism $R_y$ (or $R_y^F$ when the
metric needs to be specified)  of the tangent space $T_x(M)$, such that
$R_y=R^i_k(y)\partial_{x^i}\otimes dx^k:T_x M\rightarrow T_x M$, where
\begin{equation}
R^i_k(y)=2\partial_{x^k}G^i-y^j\partial^2_{x^j y^k}G^i+2G^j\partial^2_{y^j y^k}G^i
-\partial_{y^j}G^i\partial_{y^k}G^j.
\end{equation}
It is easily seen that the Riemannian curvature $R_y$ is self-adjoint with respect to $\langle\cdot,\cdot\rangle_y$.

Using the Riemannian curvature, we can generalize the notion of sectional curvature  in Riemannian geometry to  that of  flag curvature in Finsler geometry. Let $y$ be a nonzero tangent vector in
$T_x M$ and $\mathbf{P}$ a 2-dimensional subspace (called tangent plane) containing $y$.
Assume that $\mathbf{P}$ is linearly spanned by $y$ and $v$. Then the flag curvature
of the flag $(x, y, \mathbf{P})$   is defined to be
\begin{equation}\label{flag}
K(x,y,\mathbf{P})=\frac{\langle R_y v,v\rangle_y}
{\langle y,y\rangle_y \langle v,v\rangle_y-\langle y,v\rangle_y^2}.
\end{equation}
The flag curvature may also be denoted as $K(x,y, y\wedge v)$, or $K^F(\cdot,\cdot,\cdot)$ when the metric needs to be specified. It is obvious that
the definition of (\ref{flag}) does not depends on the choice of the nonzero tangent vector $v$ in $\mathbf{P}$.

A nowhere zero vector field  $Y$ on an open subset $\mathcal{U}$ of $M$ is called  a geodesic field on $\mathcal{U}$,  if the integration
curves of $Y$ are geodesics of nonzero constant speed. Shen proves the following useful  theorem (see \cite{CS}).
\begin{theorem}\label{shen-thm}
Let $Y$ be a geodesic field on an open subset $\mathcal{U}\subset M$ and suppose  that
$y=Y(x)\neq 0$,  for any $x\in \mathcal{U}$. Then the Riemannian curvature $R_y^F$ for $F$ equals the Riemannian curvature
${R}_y^{g_Y}$ for the Riemannian metric $g_Y$ on $\mathcal{U}$.
\end{theorem}

\section{Homogeneous $(\alpha,\beta)$-spaces with vanishing S-curvature}
Let $G$ be a connected simply connected Lie group and $H$ a compact connected subgroup of $G$,
with  Lie algebras as $\mathfrak{g}$ and $\mathfrak{h}$,
respectively. Assume that $F$ is a $G$-invariant
non-Riemannian $(\alpha,\beta)$-metric on the smooth coset space $M=G/H$.
In our previous consideration \cite{DX1}, it is proved that   $F$ can be presented as $F=\alpha\phi(\beta/\alpha)$, in which
both the Riemannian metric $\alpha$ and the $1$-form $\beta$ on $M$ are
$G$-invariant. Restricted to  $\mathfrak{m}=\mathfrak{g}/\mathfrak{h}$,
which can be identified with the tangent space $T_{eH}M$, the metric $\alpha$
defines an inner product $\langle\cdot,\cdot\rangle$, and $\beta$ defines
the linear functional $\beta(\cdot)=\langle\cdot,v\rangle$, for some nonzero $v\in\mathfrak{m}$. For simplicity we denote their restrictions to $\mathfrak{m}$  as
$\alpha$ and $\beta$, respectively. Because both $\alpha$ and $\beta$ are
$\mathrm{Ad}(H)$-invariant, $\mathfrak{m}$ can be identified with the
$\alpha$-orthogonal complement of $\mathfrak{h}$ in $\mathfrak{g}$, that is,
we have the $\alpha$-orthogonal decomposition $\mathfrak{g}=\mathfrak{h}+
\mathfrak{m}$ which satisfies $[\mathfrak{h},\mathfrak{m}]\subset\mathfrak{m}$. Then the $\alpha$-dual $v$
of $\beta$ can be viewed as a vector in $\mathfrak{g}$ with $\mathrm{Ad}(H)v=v$, i.e., $[\mathfrak{h},v]=0$. We denote the global $\alpha$-dual vector field of $\beta$ as $V$.

In \cite{ChengShen2009}, the authors give a formula of the S-curvature for an $(\alpha,\beta)$-metric
in a local coordinate system $x=(x^i)\in M$ and $y=y^j\partial_{x^i}\in TM$. In the homogeneous case,  where $\beta$ (or equivalently $V$) has constant
$\alpha$-length $b>0$, the formula can be simplified as
\begin{equation}\label{s-curvature-formula-1}
S=\alpha^{-1}\frac{\Phi}{\Delta^2}(r_{00}-2\alpha Qs_0),
\end{equation}
where
$$Q=\frac{\phi'}{\phi-s\phi'}, \Delta=1+sQ+(b^2+s^2)Q',$$
$$\Phi=-(Q-sQ')(n\Delta+1+sQ)-(b^2-s^2)(1+sQ)Q'',$$
and $r_{00}$ and $s_0$ are only relevant to $\alpha$ and $\beta$. As pointed out in \cite{ChWW}, in this case,
$\Phi$ is a nonzero scalar function times the norm of the mean Cartan tensor at any $(x,y)$.
Therefore, if  $\Phi$ vanishes identically, then  $F$ must be a  Riemannian metric.

 If the S-curvature of the non-Riemannian homogeneous $(\alpha,\beta)$-metric
$F$ on $G/H$ vanishes identically,  and the corresponding $\Phi$ is not constantly $0$, then by Theorem 1.2 of \cite{ChengShen2009}, only the third case can happen, i.e., $V$
is a Killing vector field of constant length with respect to $\alpha$, which implies that $V$ is
a Killing vector field of constant length with respect to $F$. On the other hand, if  $V$ is a Killing vector field of
constant length with respect to $F$, then it is also a Killing vector field of constant length with respect to $\alpha$. Then by Theorem 1.2 of
\cite{ChengShen2009} again,  the S-curvature  of $F$ vanishes identically.

Meanwhile, in \cite{DengWang}, S. Deng and X. Wang gave another presentation of
(\ref{s-curvature-formula-1}) for a homogeneous space $M=G/H$ without using local coordinates, namely,   for any nonzero tangent vector $y\in\mathfrak{m}=T_{eH}M$ at $(eH,y)$, we have
\begin{equation}\label{s-curvature-formula-2}
S(eH,y)=-\frac{1}{\alpha(y)}\frac{\Phi}{2\Delta^2}(-b\langle[v,y]_\mathfrak{m},y\rangle
-\alpha(y)Q\langle[v,y]_\mathfrak{m},v\rangle).
\end{equation}
From this formula it follows that,  if $\langle[v,y]_\mathfrak{m},y\rangle=\langle[v,y]_\mathfrak{m},v\rangle=0$,
for any $y\in\mathfrak{m}$,
then $S(eH,y)=0$ for any $y\in \mathfrak{m}=T_{eH}(M)$. By the homogeneity of $M=G/H$, the S-curvature vanishes
identically.

On the other hand, assume that the S-curvature of the homogeneous
non-Riemannian $(\alpha,\beta)$-metric $F=\alpha\phi(\beta/\alpha)$ vanishes
identically and there exists  $y\in\mathfrak{m}$, such that
\begin{equation}\label{9999}
\langle[v,y]_\mathfrak{m},y\rangle\neq 0\mbox{ or }
\langle[v,y]_{\mathfrak{m}},v\rangle\neq 0.
\end{equation}
Then it is easy to check that   (\ref{9999}) is still valid with $y$ changed to $y_0=c_1y+c_2v$ with $c_1\ne 0$.
Note that there exists  $y_0$ such that $\langle y_0,v\rangle=0$ and
$\alpha(y_0)=\alpha(v)=b$. For simplicity, we can change $\beta$ by a scalar multiplication and change $\phi$ correspondingly without changing the metric $F=\alpha\phi(\beta/\alpha)$, such that  $b=1$.
Denote $y_1(t)=\sqrt{1-t^2}y_0+tv$ and
$y_2(t)=-\sqrt{1-t^2}y_0+tv$. Then for any $t\in(-1,1)$, we have
$$\alpha(y_1(t))=\alpha(y_2(t))=1, \beta(y_1(t))=\beta(y_2(t))=t.$$
Thus the values of $s$, $Q(s)$, $Q'(s)$, $Q''(s)$, $\Delta$ and $\Phi$ coincide
for $y_1(t)$ and $y_2(t)$ respectively. So
\begin{eqnarray*}
S(eH,y_1(t))+S(eH,y_2(t))&=&\frac{\Phi}{2\Delta^2}
(\langle[v,y_1(t)]_\mathfrak{m},y_1(t)\rangle+\
\langle[v,y_2(t)]_\mathfrak{m},y_2(t)\rangle)\\
&=& (1-t^2)\frac{\Phi}{\Delta^2}\langle[v_0,y_0]_\mathfrak{m},y_0\rangle,
\end{eqnarray*}
which by our assumption, must be $0$ for any $t\in (-1,1)$.
 Since $F$ is non-Riemannian, $\Phi(t)=\Phi(s(y_1(t)))=\Phi(s(y_2(t)))$  is not constantly $0$. Hence we have $\langle [v_0,y_0]]_\mathfrak{m},y_0\rangle=0$.
Then by (\ref{9999}), $\langle[v_0,y_0]_\mathfrak{m},v_0\rangle\neq 0$.
Since  $S(y_1(t))\equiv 0$ for any $t\in (-1,1)$, we have
\begin{equation}\label{9998}
(\phi(t)-t\phi'(t))(t+Q(s(y_1(t))))=t\phi(t)+(1-t^2)\phi'(t)\equiv 0
\end{equation}
on the nonempty open sub-interval of $t\in (-1,1)$
such that $\Phi((t))\neq 0$. Taking the derivative of (\ref{9998}) with respect to $t$, we get
$$\phi(t)-t\phi'(t)+(1-t^2)\phi''(t)\equiv 0,$$
on some nonempty open interval. This is a contradiction with the
condition (\ref{(2.3)}) for $(\alpha,\beta)$-metrics.

To summarize, we have proven the following proposition.
\begin{proposition}
Assume that $F=\alpha\phi(\beta/\alpha)$ is a non-Riemannian homogeneous $(\alpha,\beta)$-metric
on the homogeneous space $M=G/H$. Then the following statements are equivalent:
\begin{description}
\item{\rm (1)} The S-curvature of $(M,F)$ vanishes identically.
\item{\rm (2)} The vector field $V$ is a Killing vector field of constant length with respect to $F$.
\item{\rm (3)} The $\alpha$-dual $v$ of $\beta$ in $\mathfrak{m}$ satisfies
\end{description}
\begin{equation}\label{0000}
\langle [v,y]_\mathfrak{m},y\rangle=\langle[v,y]_\mathfrak{m},v\rangle=0,\, \forall y\in\mathfrak{m},
\end{equation}
in which $[\cdot,\cdot]_{\mathfrak{m}}=\mathrm{pr}_{\mathfrak{m}}\circ[\cdot,\cdot]$.
\end{proposition}

Note that $\mathfrak{m}$ can be $\alpha$-orthogonal
decomposed as $\mathfrak{m}=\mathbb{R}v+v^{\perp}$. Applying (\ref{0000}) to a nonzero vector
$y=cv+y'\in\mathfrak{m}$ with $y'\in v^\perp$, we
have
\begin{eqnarray}
\langle[v,y'],v\rangle=0, \mbox{ and } \langle[v,y'],y'\rangle=0.\label{0001}
\end{eqnarray}
Let $K$ be the closure of the connected subgroup of $G$ corresponding to the subalgebra
$\mathfrak{h}\oplus\mathbb{R}v$, and $\mathfrak{p}=v^\perp$.
Then  by (\ref{0001}) we have
$\mathrm{Ad}(K)\mathfrak{p}\subset\mathfrak{p}$ and it is easily seen that the restriction of $\alpha$ to $\mathfrak{p}$ is $\mathrm{Ad}(K)$-invariant.
\section{Classification of positively curved homogeneous
$(\alpha,\beta)$-spaces with vanishing S-curvature}
We keep all the above notations. Assume that the homogeneous $(\alpha,\beta)$-metric
$F$ on $M=G/H$ has vanishing S-curvature and positive flag curvature.
 We first prove that $\mathfrak{h}\oplus\mathbb{R}v$ is the Lie algebra of some closed connected subgroup $K$ of $G$. More precisely, we have the following lemma.
\begin{lemma} \label{key-lemma}
Assume that  $F=\alpha\phi(\beta/\alpha)$ is a non-Riemannian
homogeneous $(\alpha,\beta)$-metric on $M=G/H$ with vanishing S-curvature and
positive flag curvature.  Then we have the following:
\begin{description}
\item{\rm (1)} The subalgebra $\mathfrak{h}\oplus\mathbb{R}v$ is
the Lie algebra of some closed connected subgroup $K$ in $G$.
\item{\rm (2)} Denote the orthogonal complement of $v$   with respect to $\alpha$ as $\mathfrak{p}=v^\perp$. Then
we have an orthogonal decomposition $\mathfrak{g}=\mathfrak{k}+\mathfrak{p}$, and the restriction of $\alpha$ to $\mathfrak{p}$ defines a homogeneous Riemannian metric $g$ on $G/K$.
\item{\rm (3)} There is a suitable constant $c>0$, such that the natural projection $\pi:(G/H,F)\rightarrow(G/K,cg)$ is a Finslerian submersion.
\end{description}
\end{lemma}
\begin{proof}
(1) Since $(M,F)$ is positively curved,  $M$ is compact, and $\mathfrak{g}$
can be decomposed as a direct sum $\mathfrak{g}=\mathfrak{g}'+\mathfrak{g}''$, where $\mathfrak{g}'$ is the Lie algebra of a compact connected group $G'$
which is the image group of natural projection from $G$ to $I_0(M,F)$, and
$\mathfrak{g}''$ is an ideal contained in $\mathfrak{h}$.  Correspondingly, we also have a direct sum decomposition $\mathfrak{h}=\mathfrak{h}'\oplus\mathfrak{h}''$. By
the positive flag curvature condition, and Lemma 5.1 of \cite{XDHH} (see also \cite{DH1}), we have
$\mathrm{rk}\mathfrak{g}'\leq\mathfrak{rk}\mathfrak{h}'+1$ (in this case,
the equality may happen). Then  the subalgebra $\mathfrak{h}'\oplus\mathbb{R}v$ contains
a Cartan subalgebra of $\mathfrak{g}$, so it is the Lie algebra of a closed
connected subgroup $K'$ in $G'$. Now the pre-image $K$ in $G$ gives the subgroup
required in the lemma.

(2) It follows immediately from the observation at the end of last section,
about the condition that S-curvature vanishes identically,  and part (1) of
this lemma.

(3) The orthogonal projection $\mathrm{pr}$ from $\mathfrak{m}=\mathbb{R}v+\mathfrak{p}$ onto $\mathfrak{p}$ induces
a unique Minkowski norm $F'$ on $\mathfrak{p}$ through the Finsler metric $F$, such that $\mathrm{pr}:(\mathfrak{m},F)\rightarrow(\mathfrak{p},F')$
is a submersion (see \cite{PD1}). Because $F$ is an $(\alpha,\beta)$-norm, it is not hard to
see that $F'$ must be a scalar multiple of $\alpha|_{\mathfrak{p}}$, which is also
$\mathrm{Ad}(K)$-invariant, and defines a $G$-invariant Riemannian metric
$F'=cg$ on $M'=G/K$ for some $c>0$. To see that the natural projection
$\pi:(G/H,F)\rightarrow(G/K,cg)$ is a global submersion, we need to check that
$\pi_*:(T_{gH}M,F(gH,\cdot)\rightarrow T_{gK}M',F'(gK,\cdot))$ is
a submersion for any $g\in G$. This follows from the facts that $\pi_*|_{gH}=(L_g)_*\circ\pi_*|_{eH}\circ(L_{g^{-1}})_*$,
that the second section of  the right side  is a submersion, and that the other two sections are linear isometries between Minkowski norms. This completes the proof of the lemma.
\end{proof}

Now we start to prove the main theorem.

Since $(G/H,F)$ has positive flag curvature, by Lemma \ref{key-lemma} and \cite{XDHH},  $(G/K,cg)$ is an even dimensional Riemannian homogeneous
space with positive sectional curvature. Then by the classification of Wallach, Aloff-Wallach and B. Bergery,  the local structure of $G/K$ can be completely determined, and from this we can deduce the classification of the coset spaces $G/H$.

To make the statement more clearly, we introduce the notion of  equivalence between  homogeneous Finsler spaces. Two homogeneous Finsler spaces $G/H$ and $G'/H'$ are called equivalent, if they are locally isometric to each other. Up to equivalence, we can assume that $G$ is compact, $H$ is connected, $G/H$ is simply connected and
$\mathfrak{h}$ does not contain any nontrivial ideal of $\mathfrak{g}$.
Moreover, we do not need to distinguish $G/H$ and $G'/H'$ if there is an isomorphism from $G$ to $G'$ which  maps $H$ onto $H'$.

Notice that $\mathfrak{k}$ may contain a nonzero ideal of $\mathfrak{g}$ corresponding  to the additional $\mathbb{R}v$-factor.  Now we prove that, in this case, this ideal must be of dimension $1$.
\begin{lemma} Let $\mathfrak{g}$ be a compact Lie algebra,  and $\mathfrak{h}$
a subalgebra of $\mathfrak{g}$ which does not contain any nontrivial ideal of $\mathfrak{g}$. Let $\mathfrak{k}=\mathfrak{h}\oplus\mathbb{R}v$ be another
subalgebra of $\mathfrak{g}$ which contain a nontrivial ideal $\mathfrak{g}'$
of $\mathfrak{g}$.  Then $\dim\mathfrak{g}'=1$.
\end{lemma}
\begin{proof}
Let $\mathfrak{g}=\mathfrak{g}_0\oplus\mathfrak{g}'$ and
$\mathfrak{k}=\mathfrak{k}_0\oplus\mathfrak{g}'$ be the direct sum decomposition with respect to a bi-invariant inner product on $\mathfrak{g}$. Then obviously $\mathfrak{k}_0\in\mathfrak{g}_0$.
If $\mathfrak{g}'$ contains any semi-simple component $\mathfrak{g}''$,
then $\mathfrak{g}''$ is an ideal of $\mathfrak{g}$ contained in $\mathfrak{h}$, which is a contradiction. So $\mathfrak{g}'$ is Abelian.
If $\dim\mathfrak{g}'>1$, then $\mathfrak{g}$ and $\mathfrak{h}$ has some
nontrivial Abelian factor, which is also a contradiction.
\end{proof}

Up to equivalence, we list all possibilities as follows. In the following, we will first present the structure in the Lie algebra level, and then give the explicit coset spaces up to equivalence.
\begin{description}
\item{\rm (1)} $\mathfrak{g}=\mathfrak{su}(n+1)$,
$\mathfrak{k}=\mathfrak{su}(n)\oplus\mathbb{R}$, and $\mathfrak{h}=\mathfrak{su}(n)$. The coset space $G/H$ is equivalent to $S^{2n+1}=\mathrm{SU}(n+1)/\mathrm{SU}(n)$, where $n\geq 1$.
\item{\rm (2)}  $\mathfrak{g}=\mathfrak{su}(n+1)\oplus\mathbb{R}$,
$\mathfrak{k}=\mathfrak{su}(n)\oplus\mathbb{R}\oplus\mathbb{R}$, and
$\mathfrak{h}=\mathfrak{su}(n)\oplus\mathbb{R}$. The coset  $G/H$ is  equivalent  to one of the coset spaces $S^{2n+1}=\mathrm{SU}(n+1)/\mathrm{SU}(n)$,  $S^{2n+1}
    =\mathrm{U}(n+1)/\mathrm{U}(n)$, for $n\geq 1$. Note that algebraically,  in this case $G/H$ can also be     covered by
$(\mathrm{SU}(n+1)/S(\mathrm{U}(n)\times\mathrm{U}(1)))\times\mathbb{R}$. However, since $(\mathrm{SU}(n+1)/S(\mathrm{U}(n)\times\mathrm{U}(1)))\times\mathbb{R}$  can not be
endowed with an invariant Riemannian metric with positive curvature, these coset spaces must be excluded.
\item{\rm (3)} $\mathfrak{g}=\mathfrak{sp}(n+1)$,
$\mathfrak{k}=\mathfrak{sp}(n)\oplus\mathbb{R}$, and
$\mathfrak{h}=\mathfrak{sp}(n)$. The coset space
$G/H$ is equivalent to $S^{4n+3}=\mathrm{Sp}(n+1)/\mathrm{Sp}(n)$.
\item{\rm (4)} $\mathfrak{g}=\mathfrak{sp}(n+1)\oplus\mathbb{R}$,
$\mathfrak{k}=\mathfrak{sp}(n)\oplus\mathbb{R}\oplus\mathbb{R}$, and
$\mathfrak{h}=\mathfrak{sp}(n)\oplus\mathbb{R}$. The coset space $G/H$ is equivalent to
$S^{4n+3}=\mathrm{Sp}(n+1)/\mathrm{Sp}(n)$ or
$S^{4n+3}=(\mathrm{Sp}(n+1)\mathrm{U}(1))/(\mathrm{Sp}(n)\times\mathrm{U}(1))$. Similarly as in case (2), algebraically in the case the coset space $G/H$ can also be
 covered by $(\mathrm{Sp}(n+1)/(\mathrm{Sp}(n)\mathrm{U}(1)))\times\mathbb{R}$. But these spaces cannot be positively curved.
\item{\rm (5)}  $\mathfrak{g}=\mathfrak{sp}(n+1)\oplus\mathbb{R}$,
$\mathfrak{k}=\mathfrak{sp}(n)\oplus\mathfrak{sp}(1)\oplus\mathbb{R}$,
and $\mathfrak{h}=\mathfrak{sp}(n)\oplus\mathfrak{sp}(1)$. In this case, the universal covering space of
$G/H$ is  $\mathbb{H}\mathrm{P}^{2n+1}\times\mathbb{R}$,
hence $G/H$ can not be positively curved;
\item{\rm (6)} $\mathfrak{g}=\mathrm{su}(3)$, $\mathfrak{k}=\mathbb{R}\oplus
\mathbb{R}$  is a Cartan subalgebra of $\mathfrak{g}$, and $\mathfrak{h}=\mathbb{R}$.
The coset space $G/H$ is equivalent to $\mathrm{U}(3)/\mathrm{T}^1$, in which $T^1$ is a one-dimensional
torus in $\mathrm{U}(3)$. In the case that $T^1$ is a maximal torus of a standard $\mathrm{SU}(2)$
in $\mathrm{SU}(3)$, up to Weyl group actions, the space can be denoted as
$\mathrm{SU}(3)/\mathrm{U}(1)$. In other cases, it will be denoted as $S_{k,l}$ with
$kl(k+l)\neq 0$, where $T^1$ is generated by $\mathrm{diag}(e^{\sqrt{-1}kt},e^{\sqrt{-1}lt},e^{-\sqrt{-1}(k+l)t})$. The spaces $S_{k,l}$  are called
the Aloff-Wallach spaces in the literature (see \cite{AloffWallach1975}).
\item{\rm (7)} $\mathfrak{g}=\mathfrak{su}(3)\oplus\mathbb{R}$,
$\mathfrak{k}=\mathbb{R}\oplus\mathbb{R}\oplus\mathbb{R}$ and  $\mathfrak{h}=\mathbb{R}\oplus\mathbb{R}$.
The coset space $G/H$ is equivalent to $\mathrm{U}(3)/{T^2}$, in which $T^2$ is a two dimensional torus in the standard maximal torus of $\mathrm{U}(3)$ consisting of diagonal matrices. Note that if $T^2$ is
contained in $\mathrm{SU}(3)$, then $G/H$ is covered by $(\mathrm{SU}(3)/T^2)\times\mathbb{R}$,
which can not be positively curved.

\item{\rm (8)}  $\mathfrak{g}=\mathfrak{sp}(3)\oplus\mathbb{R}$,
$\mathfrak{k}=\mathfrak{sp}(1)\times\mathfrak{sp}(1)\times\mathfrak{sp}(1)
\oplus\mathbb{R}$,
and $\mathfrak{h}=\mathfrak{sp}(1)\times\mathfrak{sp}(1)\times\mathfrak{sp}(1)$;
in this case $G/H$ is covered by
$(\mathrm{Sp}(3)/(\mathrm{Sp}(1)\times\mathrm{Sp}(1)\times\mathrm{Sp}(1)))\times
\mathbb{R}$, so it does not admit invariant metrics of positive flag curvature;
\item{\rm (9)}  $\mathfrak{g}=\mathfrak{f}_4\oplus\mathbb{R}$,
$\mathfrak{k}=\mathfrak{so}(9)\oplus\mathbb{R}$,
$\mathfrak{h}=\mathfrak{so}(9)$. In this case the coset space $G/H$ is covered by
$(\mathrm{F}_4/\mathrm{Spin}(9))\times \mathbb{R}$, so it does not admit invariant metrics of positive flag curvature;
\item{\rm (10)} $\mathfrak{g}=\mathfrak{f}_4\oplus\mathbb{R}$,
$\mathfrak{k}=\mathfrak{so}(8)\oplus\mathbb{R}$,
$\mathfrak{h}=\mathfrak{so}(8)$. In this case the coset space $G/H$ is covered by
$(\mathrm{F}_4/\mathrm{Spin}(8))\times \mathbb{R}$, so it does not admit invariant metrics of  positive flag curvature.
\end{description}

\section{Further discussion on the list}

In the above section, we have determined all the possible coset spaces $G/H$ which may be endowed with invariant $(\alpha,\beta)$-metrics with positive flag curvature and vanishing S-curvature. It is easily seen that each of the spaces admits an invariant Riemannian metric with positive sectional curvature. To complete the proof of the main theorem, we need to determine which coset spaces admit invariant non-Riemannian ones with positive flag curvature and vanishing S-curvature. In this section we will completely settle this problem. In particular, we prove that if a coset space in the above list admits an invariant Riemannian metric with positive sectional curvature, then there exists a nonzero Killing vector field of constant length with respect to such a Riemannian metric.

First we consider the homogeneous spheres $S^{2n+1}=\mathrm{SU}(n+1)/\mathrm{SU}(n)$, $S^{2n+1}
    =\mathrm{U}(n+1)/\mathrm{U}(n)$, $S^{4n+3}=\mathrm{Sp}(n+1)/\mathrm{Sp}(n)$, $S^{4n+3}=(\mathrm{Sp}(n+1)\mathrm{U}(1))/(\mathrm{Sp}(n)\times\mathrm{U}(1))$. On these spaces, there exists invariant normal Riemannian metrics with positive curvature (see \cite{Ber}) and there exists a nonzero invariant vector field
    which are also Kiliing vector fields with respect to the Riemannian metric. Therefore, there exists an invariant non-Riemannian $(\alpha,\beta)$-metric with positive flag curvature and vanishing S-curvature. In fact, in \cite{HD}, it is proved that on each of the above coset spaces there exists invariant non-Riemannian Randers metrics with positive flag curvature and vanishing S-curvature. The same assertion also holds for the coset spaces $S_{k,l}$ with $kl(k+l)\neq 0$; see \cite{HD}.
     Therefore we only need to consider
 the homogeneous spaces $\mathrm{SU}(3)/\mathrm{U}(1)$ in (6) and $\mathrm{U}(3)/T^2$
in (7).

 Consider $M=G/H=\mathrm{SU}(3)/\mathrm{U}(1)$. The isotropy group
is given by diagonal matrices of the form $\mathrm{diag}(e^{\sqrt{-1}t},e^{-\sqrt{-1}t},1)$, $\forall t\in\mathbb{R}$.
We now show that a $G$-invariant non-Riemannian
  $(\alpha,\beta)$-metric $F$ on $G/H$ can not be positively curved.
Fix a bi-invariant inner product on $\mathfrak{g}$, and  let $\mathfrak{m}$  be  the orthogonal complement of $\mathfrak{h}$.  Then $\mathfrak{m}$ must be the direct sum of the subspace
$\mathfrak{m}_0=\mathbb{R}\sqrt{-1}\mathrm{diag}(1,1,-2)$, which must contain the $\alpha$-dual $v$ of $\beta$ in $\mathfrak{m}$, and three
root planes, $\mathfrak{m}_1=\mathfrak{g}_{\pm(e_1-e_2)}$,
$\mathfrak{m}_2=\mathfrak{g}_{\pm(e_1-e_3)}$ and
$\mathfrak{m}_4=\mathfrak{g}_{\pm(e_2-e_3)}$, where
$e_i=\sqrt{-1}E_{i,i}$. We denote  the projection from
$\mathfrak{m}$ onto  $\mathfrak{m}_i$  (with respect to the above decomposition) as $\mathrm{pr}_i$.
Notice that the $\alpha$-dual of the one-form $\beta$ defines
a nonzero Killing vector field $V$ for $(M,F)$, of constant length. Its integral curve at each point is a geodesic, i.e., $V$ is also a geodesic
field. The localization of $F$ at $V$ defines a $\mathrm{SU}(3)$-homogeneous Riemannian metric $g_V$ on $M$.

Since both $F$ and $V$ are $\mathrm{Ad(U}(1))$-invariant, so does
$g_V$. Therefore we have  $g_V^2=g_0^2+g_1^2+g_2^2+g_3^2$, where $g_i$ defines an inner product on $\mathfrak{m}_i$ such that up to scalars $g_i^2$ coincides with the restriction to $\mathfrak{m}_i$ of the bi-invariant inner product on $\mathfrak{g}=\mathfrak{su}(3)$, for $i\ne 0$. If we denote the inner product
on $\mathfrak{m}$ defined by $g_V$ as $\langle\cdot,\cdot\rangle_{g_V}$, then
$\mathfrak{m}_i$, $\forall 0\leq i\leq 3$ are orthogonal with respect to
all inner products defined by any bi-invariant Riemannian metric on $\mathrm{SU}(3)$, by $\alpha$ and by $g_V$.

Now for any nonzero $v_1\in \mathfrak{m}_0$ and any nonzero $v_1'\in\mathfrak{m}_1$, $(v_1, v_1')$ form a linearly independent commuting pair. By Lemma 3.2 of \cite{Wa}, if
$v_1$ and $v_1'$ are viewed as  tangent vectors in $T_{eH}M$, then we have
$$\langle R^{g_V}(v_1,v_1')v_1',v_1\rangle_{g_V}=
\langle U(v_1,v_1'),U(v_1,v_1')\rangle_{g_V}
-\langle U(v_1,v_1),U(v_1',v_1')\rangle_{g_V},$$
where $U(\cdot,\cdot)\in\mathfrak{m}$ is defined by
$$\langle U(u_1,u_2),u_3\rangle_{g_V}=\frac{1}{2}
(\langle[u_3,u_1]_{\mathfrak{m}},u_2\rangle_{g_V}+
\langle [u_3,u_2]_{\mathfrak{m}},u_1\rangle_{g_V}),$$
for any $u_1$, $u_2$ and $u_3$ in $\mathfrak{m}$.
Now a direct calculation shows that $U(v_1,v_1)=U(v_1,v_1')=0$.
So the sectional curvature of $g_V$ for the tangent plane generated by
$v_1$ and $v_1'$ at $eH$ vanishes. Now applying Shen's Theorem, we have
$$K^F(eH,v,v\wedge v')=0.$$
Thus the coset space $\mathrm{SU}(3)/\mathrm{U}(1)$ in (6), where $\mathrm{U}(1)$ is a
maximal torus in a standard $\mathrm{SU}(2)\subset\mathrm{SU}(3)$,
cannot be endowed with any invariant non-Riemannian $(\alpha,\beta)$ metric with positive flag curvature.

Now consider the coset space $M=\mathrm{U}(3)/T^2$, where $T^2$ is not contained in the standard subgroup $\mathrm{SU}(3)$.
As in the above case,  fix a bi-invariant inner product of $\mathfrak{u}(3)$, and let $\mathfrak{m}$
 be the orthogonal complement of $\mathfrak{h}$. Then $\mathfrak{m}$   is the direct sum of root planes,
and a one-dimensional space generated by a diagonal matrix. If the intersection between
$\mathrm{Lie}(T^2)$ and $\mathfrak{su}(3)$ is one of
$$\mathbb{R}\sqrt{-1}\mathrm{diag}(1,-1,0),\,\,\mathbb{R}\sqrt{-1}\mathrm{diag}(1,0,-1),$$
 or $$\mathbb{R}\sqrt{-1}\mathrm{diag}(0,1,-1),$$
  then
a nonzero matrix of the form
$$\sqrt{-1}\mathrm{diag}(a,a,b),\,\,\sqrt{-1}\mathrm{diag}(a,b,a),$$
 or $$\sqrt{-1}\mathrm{diag}(b,a,a),$$
 where $a\ne b$, is contained in $\mathfrak{m}$. This matrix  generates a line which must contain
the $\alpha$-dual $v$ of $\beta$ in $\mathfrak{m}$. Without losing generality, we can assume that
$v=\sqrt{-1}\mathrm{diag}(a,a,b)$. Then we can select a nonzero vector
$v'\in\mathfrak{g}_{\pm(e_1-e_2)}$ and similarly prove that the flag curvature
$K^F(eH,v,v\wedge v')$ vanishes. Therefore,  in this case $(M,F)$ cannot be endowed with an invariant non-Riemannian $(\alpha,\beta)$-metric with positive flag curvature.

In other cases, $G/H=\mathrm{U}(3)/T^2$ is diffeomorphic to a coset space of the form
$G'/H'=S_{k,l}=\mathrm{SU}(3)/T^1$, where $T^1$ is equal to the intersection of $T^2$
with $\mathrm{SU}(3)$. Upon the $\mathrm{Ad}(G')$-action, we can assume that $T^2$ is
contained in the standard Cartan subalgebra $\mathfrak{t}$ corresponding to
the maximal torus $T$ of $\mathrm{U}(3)$ consisting of all diagonal matrices, and $T^1=U_{k,l}$. Then it is proved in  \cite{AloffWallach1975} that there exists  a positively curved $G'$-invariant Riemannian metric $\alpha$ on $G'/H'$ (so that it defines an inner product on $\mathfrak{m}'$), such that with respect to $\alpha$,
$\mathfrak{m}'$ can be orthogonally decomposed as
$\mathfrak{m}'=\mathfrak{m}'_0+\mathfrak{g}_{\pm(e_1-e_2)}+
\mathfrak{g}_{\pm(e_1-e_3)}+\mathfrak{g}_{\pm(e_2-e_3)}$, where
$\mathfrak{m}'_0$ is the orthogonal complement of $\mathrm{Lie}(T^1)$
in $\mathfrak{t}\cap\mathfrak{su}(3)$ with respect to the bi-invariant
inner product, and on each root plane, $\alpha$
differs
with the bi-invariant inner product only by a scalar multiplication.
The canonical isomorphism between $\mathfrak{m}$ and
$\mathfrak{m}'$, both viewed  as the same tangent space at $x=eH=eH'$,
induces the identity map on each root plane, and sends $\mathfrak{m}'_0$
onto $\mathfrak{m}_0$ which is also contained in $\mathfrak{t}$. By this
isomorphism, $\alpha$ defines a $\mathrm{Ad}(T)$-invariant (which is $T^2$-invariant as well) inner product on $\mathfrak{m}$. Therefore the Riemannian
metric $\alpha$ is also a $G$-invariant metric on $G/H$. Since on $G/H$ there exists a
$G$-invariant nonzero Killing vector field of constant length, this Riemannian metric can be slightly perturbed to $G$-invariant
 non-Riemannian  $(\alpha,\beta)$-metric on
$M$ with positive flag curvature and vanishing S-curvature. In fact, let $\beta$ be the $\alpha$-dual of an invariant Killing vector field of constant length on the Riemannian manifold $(G/H, \alpha)$. Then for any sufficiently small positive real number $t$, $F_t=\alpha+t\beta$ is a non-Riemannian Randers metric with vanishing S-curvature (see \cite{De}). Since $\alpha$ has positive sectional curvature, by the continuity, the Randers metric $F_t$ also has positive flag curvature for sufficiently small $t$.

The proof of Theorem \ref{main1} is now completed.

\end{document}